\documentclass[12pt,reqno]{amsart}%
\usepackage{amsmath, amsfonts, amssymb, amsthm, amscd, amsbsy}
\usepackage[usenames,dvipsnames,svgnames,x11names,hyperref]{xcolor}
\usepackage{geometry}
\usepackage{enumerate}
\usepackage{graphicx}
\usepackage[pagebackref]{hyperref}
\usepackage{amsmath}
\usepackage{amsfonts}
\usepackage{amssymb}%
\providecommand{\U}[1]{\protect\rule{.1in}{.1in}}
\hypersetup{colorlinks=true,
breaklinks=true,
urlcolor=NavyBlue,
linkcolor=Fuchsia,
bookmarksopen=false,
filecolor=black,
citecolor=ForestGreen,
linkbordercolor=red
}
\newtheorem{theorem}{Theorem}[section]
\theoremstyle{plain}

\newtheorem{lemma}{Lemma}[section]

\newtheorem{proposition}{Proposition}[section]
\newtheorem{remark}{Remark}[section]

\numberwithin{equation}{section}
\allowdisplaybreaks
\AtBeginDocument{\hypersetup{pdfborder={0 0 0.01}}}
\newcommand{\rr}{\mathbb{R}}
\newcommand{\dx}{\mathrm{dx}}

\def\e{\epsilon}

\def\p{\partial}

\def\oo{\infty}

\def\R{\mathbb{R}}

\begin{document}
\title{Short proofs of refined sharp Caffarelli-Kohn-Nirenberg inequalities}
\author{Cristian Cazacu}
\address{Cristian Cazacu: $^1$Faculty of Mathematics and Computer Science \\
University of Bucha-rest\\
010014 Bucharest, Romania\\
\&
$^2$Gheorghe Mihoc-Caius Iacob Institute of Mathematical\\
Statistics and Applied Mathematics of the Romanian Academy\\
050711 Bucharest, Romania\\
\&
$^3$Member in the research grant no. PN-III-P1-1.1-TE-2019-0456\\ University Politehnica of Bucharest\\ 
060042, Bucharest, Romania
}
\email{cristian.cazacu@fmi.unibuc.ro}
\author{Joshua Flynn}
\address{Joshua Flynn: Department of Mathematics\\
University of Connecticut\\
Storrs, CT 06269, USA}
\email{joshua.flynn@uconn.edu}
\author{Nguyen Lam}
\address{Nguyen Lam: School of Science \& Environment\\
Grenfell Campus, Memorial University of Newfoundland\\
Corner Brook, NL A2H5G4, Canada }
\email{nlam@grenfell.mun.ca}
\thanks{C.C. was partially supported by CNCS-UEFISCDI Romania, Grant No. PN-III-P1-1.1-TE-2019-0456. J.F. was partially supported by a Simons Collaboration grant from the Simons Foundation. N.L. was partially supported by NSERC Discovery Grant}
\date{\today}

\begin{abstract}

This note relies mainly on a refined version of the main results of the paper \cite{CC} by F. Catrina and D. Costa. We provide very short and self-contained proofs. Our results are sharp and minimizers are obtained in suitable functional spaces.  As main tools we use the so-called \textit{expand of
squares} method to establish sharp weighted $L^{2}$-Caffarelli-Kohn-Nirenberg (CKN) inequalities and density arguments.  

\end{abstract}
\maketitle

\section{Introduction}

 Our paper is mainly motivated by the sharp $L^2$-Caffarelli-Kohn-Nirenberg (CKN) inequalities of the form 
 \begin{equation}\label{CKN}
 \int_{\rr^N} \frac{|u|^2}{|x|^{2a}} \dx \int_{\rr^N} \frac{|\nabla u|^2}{|x|^{2b}}\dx \geq C^2(N, a, b)\left(\int_{\rr^N} \frac{|u|^2}{|x|^{a+b+1}}\dx\right)^2, \quad u\in C_0^\infty(\rr^N\setminus\{0\}),
  \end{equation}
  which were established by Catrina and Wang in \cite{CW} for $a=b+1$, and Catrina and Costa in \cite[Th. 1]{CC} for general $a,  b\in \rr$. The most general family of CKN inequalities were first introduced and studied in \cite{CKN, CKN2}  as a necessity when studying well-posedness and regularity of solutions to certain Navier-Stokes equations. 
  In these pioneering works the authors did not address the sharp constant $C(N,a,b)$ (i.e., the largest constant $C(N,a,b)$ such that \eqref{CKN} is true).
  It was later when various contributors studied the problem of determining sharp constants and existence/non-existence of extremizers for general CKN inequalities with specific parameter values; see for instance \cite{CC}, \cite{CW}, \cite{Co}, \cite{DLL}, \cite{LL}.
   
  In the present paper, we refine the sharp CKN inequalities given by \eqref{CKN} established by replacing the gradient term with a term in terms of the radial derivative operator. 
  Namely, we prove the sharp refined CKN inequalities
  \begin{equation}\label{CKN_refined}
    \int_{\rr^N} \frac{|u|^2}{|x|^{2a}} \dx \int_{\rr^N} \frac{|x\cdot \nabla u|^2}{|x|^{2b+2}}\dx \geq \tilde{C}^2(N, a, b)\left(\int_{\rr^N} \frac{|u|^2}{|x|^{a+b+1}}\dx\right)^2, \quad u\in C_0^\infty(\rr^N\setminus\{0\}),
  \end{equation}
  where $\tilde{C}(N,a,b)$ denotes the sharp constant and $(a,b)\in\rr^{2}$ are again arbitrary parameters.
  In \cite{CC} the authors first proved \eqref{CKN} for certain parameter values and then, to achieve sharpness for the remaining parameter values, they relied on the technical tools of spherical harmonics and the Kelvin transform.
  Our proof provides an elementary way to simultaneously obtain the sharp inequality \eqref{CKN} for all parameter values $(a,b) \in\rr^{2}$.

  Now, it is easy to see that the sharp constant $C(N,a,b)$ in \eqref{CKN} may be obtained as the infimum of a certain energy functional:
  \begin{equation}\label{best_constant}
    C(N, a, b):=\inf_{u\in C_0^\infty(\rr^N\setminus\{0\})} E(u),\ \textrm{ where }\  E(u):=\frac{ \left(\int_{\rr^N} \frac{|u|^2}{|x|^{2a}} \dx\right)^{\frac 1 2} \left(\int_{\rr^N} \frac{|\nabla u|^2}{|x|^{2b}}\dx\right)^{\frac 1 2}}{\int_{\rr^N} \frac{|u|^2}{|x|^{a+b+1}}\dx}. 
  \end{equation} 
  Moreover, it turns out that the expression for the sharp constant and the functional form of the extremizers depend on where the parameters $(a,b)$ lie in the plane $\R^{2}$.
  Thus, following \cite{CC}, we define
  \begin{equation}
    \left\{\begin{array}{ll}
      \mathcal{A}_1:=\{(a, b) \ |\ b+1-a > 0, \ b\leq (N-2)/2\}, & \\ [5pt]
      \mathcal{A}_2:=\{(a, b) \ |\ b+1-a < 0, \ b\geq (N-2)/2\}, & \\ [5pt]
      \mathcal{A}:=\mathcal{A}_1\cup \mathcal{A}_2, & \\ [5pt]
      \mathcal{B}_1:=\{(a, b) \ |\ b+1-a < 0, \ b\leq (N-2)/2\}, & \\ [5pt]
      \mathcal{B}_2:=\{(a, b) \ |\ b+1-a > 0, \ b\geq (N-2)/2\}, & \\ [5pt]
      \mathcal{B}:=\mathcal{B}_1\cup \mathcal{B}_2 & \\ [5pt]
    \end{array}\right.
    .
  \end{equation}
  We also note that, along with $\mathcal{C} =\left\{ (a,b) \ | \ a=b+1 \right\}$, the sets $\mathcal{A}$, $\mathcal{B}$ and $\mathcal{C}$ partition the parameter plane $\R^{2}$, and that $\mathcal{C}$ is a sort of interface between $\mathcal{A}$ and $\mathcal{B}$.
  In fact, $\mathcal{C}$ is precisely where \eqref{CKN} fails to have extremizers.

  \medskip

  The main result in \cite{CC} is the following theorem (statement taken from the original paper \textit{m\'{o}t-\'{a}-m\'{o}t}):

  \begin{theorem}[Theorem 1, \cite{CC}]\label{ThCC}
    According to the location of the points $(a, b)$ in the plane, we have: 
    \begin{enumerate}[(a)]
      \item In the region $\mathcal{A}$, the best constant is $C(N, a, b)=\frac{|N-(a+b+1)|}{2}$ and it is achieved by the functions $u(x)=D \exp(\frac{t|x|^{b+1-a}}{b+1-a})$, with $t<0$ in $\mathcal{A}_1$ and $t>0$ in $\mathcal{A}_2$, and $D$ a nonzero constant.
      \item In the region $\mathcal{B}$, the best constant is $C(N, a, b)=\frac{|N-(3b-a+3)|}{2}$ and it is achieved by the functions $u(x)=D |x|^{2(b+1)-N}\exp(\frac{t|x|^{b+1-a}}{b+1-a})$, with $t>0$ in $\mathcal{B}_1$ and $t<0$ in $\mathcal{B}_2$.
      \item In addition, the only values of the parameters where the best constant is not achieved are those on the line $a=b+1$ where the best constant is $C(N, b+1, b)=\frac{|N-2(b+1)|}{2}$. 
    \end{enumerate}	
  \end{theorem}

  While the constants $C(N,a,b)$ obtained in Theorem \ref{ThCC} are sharp, the extremizers (in cases (a) and (b)) do not belong to $C_{0}^{\oo}\left( \rr^{N}\setminus\left\{ 0 \right\} \right)$ (the obtained extremizers are neither smooth enough, nor compactly supported, nor necessarily vanishing at the origin $x=0$, etc.).
  Moreover, the arguments in \cite{CC} are carried out for functions with the assumption of $C_{0}^{\oo}\left( \rr^{N}\setminus\left\{ 0 \right\} \right)$ regularity.
  Thus, in order to properly study extremal properties of \eqref{CKN}, it is necessary to introduce larger function spaces which contain $C_{0}^{\oo}\left( \rr^{N}\setminus\left\{ 0 \right\} \right)$, and so that the extremizers may be approximated by $C_{0}^{\oo}\left( \rr^{N} \setminus\left\{ 0 \right\} \right)$ functions in the appropriate norm.
  In particular, the authors in \cite{CC} introduced the energy space $H_{a, b}^1$ as the completion of $C_0^\infty(\rr^N\setminus\{0\})$ in the weighted Sobolev norm 
  \begin{equation}\label{weight_Sob}
    \|u\|_{H_{a, b}^1}:= \left(\int_{\rr^N}\left [\frac{|\nabla u|^2}{|x|^{2b}}+ \frac{|u|^2}{|x|^{2a}}\right]\dx  \right)^{1/2},
  \end{equation}   
  to serve for the study of non-trivial solutions to some classes of second order nonlinear PDE. 
  However, the space $H_{a,b}^1$ was omitted both in the statement and the original proof of Theorem \ref{ThCC}.      
  On the other hand, it is easy to see that, by density and definition of $H_{a,b}^{1}$, the inequality \eqref{CKN} holds for functions in $H_{a,b}^{1}$, and, moreover, by density and Fatou's lemma, there holds 
  $$C(N, a, b)=\inf_{u \in C_0^\infty(\rr^N\setminus\{0\}), u\neq 0} E(u)=\inf_{u \in H_{a, b}^1, u\neq 0} E(u).$$ 
  However, from our point of view, it is not clear by default that the minimizers emphasized in Theorem \ref{ThCC} belong to the energy space $H_{a, b}^1$. Of course, if $u$ was a minimizer from Theorem \ref{ThCC}, then one could easily check that the norm $\|u\|_{H_{a, b}^1}$ is finite; however, one also needs to ensure that $u$ can be approximated by $C_0^\infty\left( \rr^{N} \setminus\left\{ 0 \right\} \right)$ functions in the $H_{a, b}^1$-norm.    
  In this paper, we explicitly show that the extremizers of our refinement \eqref{CKN_refined} of \eqref{CKN} may be approximated by $C_{0}^{\oo}\left( \rr^{N} \setminus \left\{ 0 \right\} \right)$ functions with respect to the appropriate weighted Sobolev norm.
  It will be easy to see that this argument may be performed for the extremizers of \eqref{CKN} with respect to the $H_{a,b}^{1}$ norm. 
  \\ 

  Our main contributions are stated in Section \ref{sec_results} and proved in Section \ref{proofs}.  
  \subsection*{Novelty of the paper}
  To resume, the new insights of this paper are mainly based on the following aspects:
  \begin{itemize}
    \item We improve the sharp CKN inequalities \eqref{CKN}  from \cite{CC} by relaxing the gradient term $|\nabla u|$ with the radial derivate term $|\partial_r u|=\frac{|x\cdot \nabla u|}{|x|}$ (see \eqref{CKN_refined}).
    \item We provide a very short and compact proof of the refined CKN inequalities.
      As a consequence, this provides an elementary and spherical harmonic-free proof for the results in \cite{CC}.
    \item We rigorously define the notion of minimizers and show their membership to suitable weighted Sobolev spaces.    
  \end{itemize}

  \section{Main results}\label{sec_results}

  In order to establish the proper functional setting for our main results we introduce the space $\mathcal{H}_{a, b}^1$ defined by
  \begin{equation}\label{weight_Sob2}
    \mathcal{H}_{a, b}^1:=\overline{C_0^\infty(\rr^N\setminus\{0\})}^{\|\cdot\|_{\mathcal{H}_{a, b}^1}} , \  \textrm{ where }\|u\|_{\mathcal{H}_{a, b}^1}:= \left(\int_{\rr^N}\left [\frac{|x\cdot \nabla u|^2}{|x|^{2b+2}}+ \frac{|u|^2}{|x|^{2a}}\right]\dx  \right)^{1/2};
  \end{equation}   
  i.e., $\mathcal{H}_{a,b}^{1}$ is the completion of $C_0^\infty(\rr^N\setminus\{0\})$ with respect to the weighted Sobolev norm  $\|\cdot\|_{\mathcal{H}_{a, b}^1}$. 

  Next we remark that the best constant in the refined $L^2$-CKN inequalities \eqref{CKN_refined} may be computed as
  \begin{equation}\label{best_constant_refined}
    \tilde{C}(N, a, b):=\inf_{u\in C_0^\infty(\rr^N\setminus\{0\}), u\neq 0} \tilde{E}(u),\ \textrm{ where } \tilde{E}(u):=\frac{ \left(\int_{\rr^N} \frac{|u|^2}{|x|^{2a}} \dx\right)^{\frac 1 2} \left(\int_{\rr^N} \frac{|x\cdot \nabla u|^2}{|x|^{2b+2}}\dx\right)^{\frac 1 2}}{\int_{\rr^N} \frac{|u|^2}{|x|^{a+b+1}}\dx}. 
  \end{equation} 
  Of course, we also have 
  $\tilde{C}(N, a, b):=\inf_{u\in \mathcal{H}_{a, b}^1, u\neq 0} \tilde{E}(u)$. 

  We may now state our main result which refines Theorem \ref{ThCC}.

  \begin{theorem}\label{ThCC_refined}
    According to the location of the points $(a, b)$ in the plane, we have: 
    \begin{enumerate}[(a)]
      \item In the region $\mathcal{A}$, the best constant is $\tilde{C}(N, a, b)=\frac{|N-(a+b+1)|}{2}$ and it is achieved in $\mathcal{H}_{a, b}^1$ by the functions 
	\begin{equation}\label{min1}
	  u_{\mathcal{A}}(x)=D\left(\frac{x}{|x|}\right) \exp\left(\frac{t|x|^{b+1-a}}{b+1-a}\right),	\end{equation}
	with $t<0$ in $\mathcal{A}_1$ and $t>0$ in $\mathcal{A}_2$, and $D:S^{N-1}\rightarrow \rr$ is such that $x\mapsto D\left(\frac{x}{|x|}\right)\in C^1(\rr^N\setminus\{0\})$.  
      \item In the region $\mathcal{B}$, the best constant is $\tilde{C}(N, a, b)=\frac{|N-(3b-a+3)|}{2}$ and it is achieved in $\mathcal{H}_{a, b}^1$ by the functions 
	\begin{equation}\label{min2}
	  u_{\mathcal{B}}(x)=D\left(\frac{x}{|x|}\right) |x|^{2(b+1)-N}\exp\left(\frac{t|x|^{b+1-a}}{b+1-a}\right), 
	\end{equation}
	with $t>0$ in $\mathcal{B}_1$ and $t<0$ in $\mathcal{B}_2$, and $D:S^{N-1}\rightarrow \rr$ is such that  $x\mapsto D\left(\frac{x}{|x|}\right)\in C^1(\rr^N\setminus\{0\})$.
      \item In the degenerate case $a=b+1$, the CKN type inequality \eqref{CKN_refined} reduces to the weighted Hardy inequality
	\begin{equation}\label{Hardy_refined}
	  \int_{\rr^N} \frac{|x\cdot \nabla u|^2}{|x|^{2b+2}}\dx \geq \tilde{C}^2(N, b+1, b) \int_{\rr^N} \frac{|u|^2}{|x|^{2b+2}}\dx.
	\end{equation}
	The best constant is $\tilde{C}(N, b+1, b)=\frac{|N-2(b+1)|}{2}$ and is not achieved in $\mathcal{H}_{b+1, b}^1$. 
    \end{enumerate}	
  \end{theorem}

  \begin{remark}\rm  Theorem \ref{ThCC_refined} implies Theorem \ref{ThCC}. 
  \end{remark}  
  Indeed, the diffeomorphism $\mathcal{L}: (0, \infty)\times S^{N-1}\rightarrow\rr^N\setminus\{0\}$, $\mathcal{L}(r, \sigma)=r \sigma$ allows us to write in spherical coordinates:
  $$x= r \sigma, \quad r:=|x|, \quad \sigma:=\frac{x}{|x|}.$$
  Then the radial derivate operator is defined by 
  $$\mathcal{R}:=\partial_r= \frac{x}{|x|}\cdot \nabla$$	 
  and so 
  $$\|u\|_{\mathcal{H}_{a, b}^1}= \left(\int_{\rr^N}\left [\frac{|\mathcal{R} u|^2}{|x|^{2b}}+ \frac{|u|^2}{|x|^{2a}}\right]\dx  \right)^{1/2}.$$
  Since
  $$|\mathcal{R} u|\leq |\nabla u|, \quad \forall u \in C_0^\infty(\rr^N\setminus\{0\}),$$
  we have $\|u\|_{\mathcal{H}_{a, b}^1}\leq \|u\|_{H_{a, b}^1}$ and therefore $H_{a, b}^1 \subset \mathcal{H}_{a, b}^1$.

  As a consequence of this, it is straightforward to see that Theorem \ref{ThCC_refined} implies Theorem \ref{ThCC} by taking the functions $D=D(x/|x|)$ in the expressions of the minimizers as constant functions in which case we have $|\nabla u|=|\mathcal{R} u|$ and $\|u\|_{\mathcal{H}_{a, b}^1}= \|u\|_{H_{a, b}^1}$.   	
  \begin{remark}\rm
    Observe that the minimizers of inequality \eqref{CKN_refined} in Theorem \ref{ThCC_refined} are not necessary radially symmetric, which is in contrast to the minimizers of inequality \eqref{CKN} since the minimizers of \eqref{CKN} are always radially symmetric. 
  \end{remark}	 

  \section{Proof of Theorem \ref{CKN_refined}}\label{proofs}    
  We prove Theorem \ref{CKN_refined} in four main steps.  The first step in the proof is the following lemma.   
  \subsection{A key inequality}
  \begin{lemma}\label{key-estimates}
    We have that 
    \begin{multline}\label{our_CKN}
      \int_{\rr^N} \frac{|u|^2}{|x|^{2a}} \dx \int_{\rr^N} \frac{|x\cdot \nabla u|^2}{|x|^{2b+2}}\dx \\ \geq \min\left\{\frac{|N-(a+b+1)|^2}{4}, \frac{|N-(3b-a+3)|^2}{4} \right\} \left(\int_{\rr^N} \frac{|u|^2}{|x|^{a+b+1}}\dx\right)^2, 
    \end{multline}
    for any  $a, b\in \rr$  and any $u\in C_0^\infty(\rr^N\setminus\{0\})$. In consequence, 
    \begin{equation}\label{key_const}\tilde{C}^2(N, a, b)\geq \min\left\{\frac{|N-(a+b+1)|^2}{4}, \frac{|N-(3b-a+3)|^2}{4} \right\}.
    \end{equation}
    (In fact, we will see later that equality holds in \eqref{key_const}.)
  \end{lemma}
  \begin{proof}[Proof of Lemma \ref{key-estimates}]

    We apply the \textit{expanding the square} method for three terms as follows (note that, in \cite{Co}, the authors used only two terms in the expansion). Namely, 
    the square that will be expanded takes the form
    \[
      S:=\int_{\rr^N}\left||x|^{-b-1}x\cdot \nabla{u}+t|x|^{-a}{u}+s|x|^{-b-1}u\right|^{2} \dx,
    \]
    where $s,t$ are real parameters to be chosen. 

    Expanding $S$ we get
    \begin{align*}
      S &  =\int_{\rr^N}\frac{|x\cdot \nabla {u}|^{2}}{|x|^{2b+2}} +t^{2}\frac{|{u}|^{2}}{|x|^{2a}%
    }+s^{2}\frac{|{u}|^{2}}{|x|^{2b+2}}\dx\\
    &  +2t\int_{\rr^N}\frac{{u}}{|x|^{a+b+1}}x\cdot\nabla{u}\dx +2s\int_{\rr^N}\frac{{u}}{|x|^{2b+2}%
  }x\cdot\nabla{u}\dx+2st\int_{\rr^N}\frac{|{u}|^{2}}{|x|^{a+b+1}}\dx.
\end{align*}
Using the divergence theorem, we have for any real $d$ that
\[%
  {\displaystyle\int_{\rr^N}}
  \frac{1}{\left\vert x\right\vert ^{d+1}}u\left(  x\cdot\nabla u\right)
  \dx=-\frac{\left(  N-d-1\right)  }{2}%
  {\displaystyle\int_{\rr^N}}
  \frac{\left\vert u\right\vert ^{2}}{\left\vert x\right\vert ^{d+1}}\dx.
\]
Combining these computations we obtain
\begin{align*}
  S &  =\int_{\rr^N}\frac{|x\cdot \nabla{u}|^{2}}{|x|^{2b+2}}+t^{2}\frac{|{u}|^{2}}{|x|^{2a}%
}+s^{2}\frac{|{u}|^{2}}{|x|^{2b+2}}\dx\\
&  -\left(  N-a-b-1\right)  t\int_{\rr^N}\frac{\left\vert u\right\vert ^{2}%
}{|x|^{a+b+1}}\dx\\
&  -\left(  N-2b-2\right)  s\int_{\rr^N}\frac{\left\vert u\right\vert ^{2}%
}{|x|^{2b+2}}\dx\\
&  +2st\int_{\rr^N}\frac{|{u}|^{2}}{|x|^{a+b+1}}\dx.
\end{align*}
Or
\begin{align}\label{imp_Id}
  S  & =\int_{\rr^N}\frac{|x\cdot \nabla u|^{2}}{|x|^{2b+2}}\dx +t^{2}\int_{\R^{N}}\frac{|{u}|^{2}}{|x|^{2a}%
}\dx+t\left[  2s-\left(  N-a-b-1\right)  \right]  \int_{\rr^N}\frac{\left\vert
  u\right\vert ^{2}}{|x|^{a+b+1}}\dx\nonumber\\
  & +s\left[  s-\left(  N-2b-2\right)  \right] \int_{\rr^N} \frac{|{u}|^{2}}{|x|^{2\left(
  b+1\right)  }}\dx.
\end{align}
Observe that all the terms that we need are in the first line of \eqref{imp_Id}.  Therefore we need to choose $s=0$ or $s=N-2b-2$ to make the last term
disappear.

\textit{The case $s=0$.} Then we get 
\begin{align}\label{relation_1}
  0\leq S &=\int_{\rr^N}\left||x|^{-b-1}x\cdot \nabla{u}+t|x|^{-a}{u}\right|^{2} \dx \nonumber\\
  &=t^{2}\int_{\R^{N}}\frac{|{u}|^{2}}{|x|^{2a}%
}\dx-t\left(  N-a-b-1\right)  \int_{\rr^N}\frac{\left\vert
  u\right\vert ^{2}}{|x|^{a+b+1}}\dx + \int_{\rr^N}\frac{|x\cdot \nabla{u}|^{2}}{|x|^{2b+2}}\dx, \nonumber\\
  &:=A t^2 -(N-a-b-1)B t + C, \quad \forall t\in \rr.
\end{align}
From \eqref{relation_1} we must have non-positive discriminant, i.e.,  $\delta:= (N-a-b-1)^2B^2-4AC\leq 0$ which is equivalent to  
\begin{equation}\label{our_CKN_1}
  \int_{\rr^N} \frac{|u|^2}{|x|^{2a}} \dx \int_{\rr^N} \frac{|x\cdot \nabla u|^2}{|x|^{2b+2}}\dx \geq \frac{|N-(a+b+1)|^2}{4} \left(\int_{\rr^N} \frac{|u|^2}{|x|^{a+b+1}}\dx\right)^2.
\end{equation}
The equality in \eqref{our_CKN_1} (i.e. $\delta=0$) implies  
\begin{equation}\label{equality}
  S=0 \textrm{ with the double root } t:=\frac{(N-a-b-1)}{2}\frac{\int_{\R^{N}} \frac{|u|^2}{|x|^{a+b+1}}\dx }{\int_{\R^{N}} \frac{|u|^2}{|x|^{2a}}\dx}.  
\end{equation} 
In fact, this extends to any $t$ such that $\textrm{sgn } t(N-a-b-1)=1$ unless $a=b+1$, since  \eqref{our_CKN_1} is invariant under dilatations (if $u$ is a minimizer also $x\mapsto u(\lambda x)$ for any $\lambda\neq 0$ is a minimizer). 

\textit{The case $s=N-2b-2$.} Then from \eqref{imp_Id} we get 
\begin{multline}\label{relation_2}
  0\leq S=\int_{\rr^N}\left||x|^{-b-1}x\cdot \nabla{u}+t|x|^{-a}{u} + (N-2b-2)|x|^{-b-1} u\right|^{2} \dx \\
  =\int_{\rr^N}\frac{|x\cdot \nabla{u}|^{2}}{|x|^{2b+2}}\dx +t^{2}\int_{\R^{N}}\frac{|{u}|^{2}}{|x|^{2a}%
}\dx +t\left(  N-3b+a-3\right)  \int_{\rr^N}\frac{\left\vert
  u\right\vert^{2}}{|x|^{a+b+1}}\dx, \quad \forall t\in \rr
\end{multline}
Similarly as in the case $s=0$,  \eqref{relation_2} implies 
\begin{equation}\label{our_CKN_2}
  \int_{\rr^N} \frac{|u|^2}{|x|^{2a}} \dx \int_{\rr^N} \frac{|x\cdot \nabla u|^2}{|x|^{2b+2}}\dx \geq \frac{|N-(3b-a+3)|^2}{4} \left(\int_{\rr^N} \frac{|u|^2}{|x|^{a+b+1}}\dx\right)^2.
\end{equation}
The equality in \eqref{our_CKN_2} implies  that (from \eqref{relation_2}) 
\begin{equation}\label{equality2}
  S=0 \textrm{ with the double root } t:=-\frac{(N-3b+a-3)}{2}\frac{\int_{\R^{N}} \frac{|u|^2}{|x|^{a+b+1}}\dx }{\int_{\R^{N}} \frac{|u|^2}{|x|^{2a}}\dx}.  
\end{equation} 
In fact, this extends to any $t$ such that $\textrm{sgn } t(N-3b+a-3)=1$ unless $a=b+1$, since  \eqref{our_CKN_2} is invariant under dilatations (if $u$ is a minimizer also $x\mapsto u(\lambda x)$ for any $\lambda\neq 0$ is a minimizer).
Consequently, combining \eqref{our_CKN_1}-\eqref{our_CKN_2} we finally obtain \eqref{our_CKN} and the proof of this lemma is complete.
\end{proof} 

The second step in the proof consists of showing that the right hand side in \eqref{our_CKN} coincides with the best constant $\tilde{C}(N, a,b)$ claimed in Theorem \ref{ThCC_refined} for all parameter values.
\subsection{The constant $\tilde{C}(N, a, b)$}
Let us establish explicitly the constant in the right hand side in \eqref{our_CKN}. For that we have to see when 
$|N-(a+b+1)|\geq |N-(3b-a+3)|$ or the reverse.
Easy computations lead to 
\begin{multline}
  T:=|N-(a+b+1)|^2-|N-(3b-a+3)|^2=(b-a+1)(N-2-2b)
\end{multline} 
We note that $T\geq 0$ when $(a, b)\in \mathcal{B}$ and $T\leq 0$ when $(a, b)\in \mathcal{A}$.
So, 	\begin{equation}\label{explicit}\min\left\{\frac{|N-(a+b+1)|^2}{4}, \frac{|N-(3b-a+3)|^2}{4} \right\}=\left\{\begin{array}{cc}
    \frac{|N-(a+b+1)|^2}{4}, & \textrm{ when } (a, b)\in \mathcal{A} \\[6pt]
    \frac{|N-(3b-a+3)|^2}{4}, & \textrm{ when } (a, b)\in \mathcal{B}
  \end{array}\right. .
\end{equation}
Notice that the constant in \eqref{explicit} is always positive. 

In the degenerate case $a=b+1$,  in which case $(a, b)\not \in \mathcal{A}\cup \mathcal{B}$, we obtain $T=0$. So, in this case the two constants are equal to $\frac{|N-2(b+1)|}{2}$. \hfill	$\square$ \\

\medskip
The third step in the proof is to determine the candidate minimizers for the best constant $\tilde{C}(N, a,b)$. 

\subsection{Determination of minimizers}\label{minimizers}

\textit{The case $\tilde{C}(N, a, b)=\left|\frac{N-(a+b+1)}{2}\right|$}.  In view of \eqref{our_CKN_1} the possible minimizers must satisfy the first order PDE
$$|x|^{-b-1}x\cdot \nabla{u}+t|x|^{-a}{u}=0, \quad x\neq 0,$$
for some real value $t$ as in \eqref{equality}. Dividing by $u\neq 0$ we get 
\begin{equation}\label{pde1}
  \frac{\partial_r u}{u}=-t  |x|^{b-a}=-t r^{b-a}.
\end{equation}
$\bullet$  If $b-a+1\neq 0$ then our PDE is equivalent to 
$$\partial_r \left(\log |u|+ t \frac{r^{b-a+1}}{b-a+1}\right)=0.$$
This implies $\log |u|+ t \frac{r^{b-a+1}}{b-a+1}=C_2(\sigma)$ (constant in $r$) from where we obtain the solutions $$u(x)=D\left(\frac{x}{|x|}\right) \exp \left(-\frac{t}{b-a+1}|x|^{b-a+1}\right).$$

Direct computations lead to  $\|u\|_{\mathcal{H}_{a, b}^1}<\infty$ if $\textrm{sgn}(\frac{t}{b-a+1})=1$, $(a, b)\in \mathcal{A}$ and $x\mapsto D (\frac{x}{|x|})\in C^1(\rr^N\setminus\{0\})$.   

$\bullet$ Otherwise, if $b-a+1=0$ then \eqref{pde1} reduces to 
\begin{equation}\label{ODE1}
  \frac{\partial_r u}{u}=-\frac{t}{r},
\end{equation}  
or equivalently $$\partial_r (\log |u|+t \log r)=0.$$
Notice that in view of \eqref{our_CKN_1} $t$ becomes an explicit constant, namely $t=\frac{N-2b-2}{2}$.  
This implies $\log (|u|r^{t})=C_1(\sigma)$ (constant with respect to $r$) and so $u(x)=\alpha(x/|x|) |x|^{-t}$ for some function $\alpha$ depending only on the spherical part.
By direct computation, we obtain $\|u\|_{\mathcal{H}_{b+1, b}^1}=\infty$ and there are not minimizers in this case. 

In fact, the case $a=b+1$ leads to a degenerate inequality in  \eqref{our_CKN} which reduces to the weighted Hardy type inequality 
$$  \int_{\rr^N} \frac{|x\cdot \nabla u|^2}{|x|^{2b+2}}\dx \geq \tilde{C}^2(N, b+1, b)\int_{\rr^N} \frac{|u|^2}{|x|^{2(b+1)}}\dx.$$
According to \eqref{relation_1} and the discussion above we obtain 
\begin{equation}\label{ident_Hardy}
  \int_{\rr^N} \frac{|x\cdot \nabla u|^2}{|x|^{2b+2}}\dx - \tilde{C}^2(N, b+1, b)\int_{\rr^N} \frac{|u|^2}{|x|^{2(b+1)}}\dx= \int_{\R^{N}} \left|\partial_r \left(u |x|^{\frac{N-2b-2}{2}}\right)\right|^2 |x|^{2-N} \dx.
\end{equation}
Due to \eqref{ident_Hardy} we obtain that  $\tilde{C}(N, b+1, b)=\frac{|N-2(b+1)|}{2}$ is indeed the best constant. To argue, for $ \epsilon> 0$ small enough, we can take a minimizing sequence $\{u_\epsilon\}_{\epsilon>0}\subset C_0^\infty(\rr^N\setminus\{0\})$ of the form 
$$u_\epsilon(x)=|x|^{-\frac{N-2b-2}{2}}\phi_\epsilon(x), $$
where $\theta_\epsilon\in C_c^\infty(\rr^N\setminus\{0\})$ is a cut-off function as in \eqref{reg_apr_seq}.  
Since, as $ \epsilon\searrow 0$, we have
$$\int_{\R^{N}} \left|\partial_r \left(u_\epsilon |x|^{\frac{N-2b-2}{2}}\right)\right|^2 |x|^{2-N} \dx=O\left(\left(\log \frac{1}{\epsilon}\right)^{-1}\right) \textrm{ and } \int_{\rr^N} \frac{|u_\epsilon|^2}{|x|^{2(b+1)}}\dx=O\left(\log\frac{1}{\epsilon}\right). $$
In view of \eqref{ident_Hardy} we have that $\{u_\epsilon\}_{\epsilon>0}$ is a minimizing sequence for the constant $\tilde{C}(N, b+1, b)=\frac{|N-2(b+1)|}{2}$. The details are left to the reader.
\\
\textit{The case $C(N, a, b)=\left|\frac{N-(3b-a+3)}{2}\right|$}.  In view of \eqref{relation_2} the possible minimizers must satisfy the first order PDE
$$|x|^{-b-1}x\cdot \nabla{u}+t|x|^{-a}{u}+(N-2b-2) u |x|^{-b-1}=0, \quad x\neq 0,$$
for some real value $t$. Dividing by $u\neq 0$ we get 
\begin{equation}\label{pde2}
  \frac{\partial_r u}{u}=-t |x|^{b-a}-(N-2b-2) \frac{1}{|x|}.
\end{equation}
$\bullet$ If $b-a+1\neq 0$ then our PDE \eqref{pde2} is equivalent to 
$$\partial_r \left(\log |u|+ t \frac{r^{b-a+1}}{b-a+1}+(N-2b-2)\log r\right)=0.$$
This implies $\log |u|+ t \frac{r^{b-a+1}}{b-a+1} +(N-2b-2)\log r=D_2(\sigma)$ (constant in $r$) from where we obtain the solutions $$u(x)=D\left(\frac{x}{|x|}\right)|x|^{-(N-2b-2)} \exp \left(-\frac{t}{b-a+1}|x|^{b-a+1}\right).$$ Direct computations lead to  $\|u\|_{\mathcal{H}_{a, b}^1}<\infty$   if $\textrm{sgn}(\frac{t}{b-a+1})=1$, $(a, b)\in \mathcal{B}$ and  $x\mapsto D (\frac{x}{|x|})\in C^1(\rr^N\setminus\{0\})$. \\

$\bullet$ Otherwise, if $b-a+1=0$ then this reduces to 
\begin{equation}\label{ODE2}
  \frac{\partial_r u}{u}=-(t+N-2b-2) \frac{1}{r}.
\end{equation} 
Notice that in view of \eqref{equality2}, $t$ becomes an explicit constant, namely $t=-\frac{N-2b-2}{2}$ and therefore \eqref{ODE2} coincides with \eqref{ODE1}. Thus, the same analysis as in the previous case and the same results apply here (since also $\left|\frac{N-(a+b+1)}{2}\right|=\left|\frac{N-(3b-a+3)}{2}\right|$ when $a=b+1$).  \hfill	$\square$
\endproof  
\medskip
To complete the proof of Theorem \ref{ThCC} we need to show that the minimizers  \eqref{min1} and \eqref{min2} (obtained in Subsection \ref{minimizers} above)  belong to  $\mathcal{H}_{a, b}^1$. This will be shown in the following subsection by an approximation (density) argument.  

\subsection{Density in an intermediate space}

In the fourth step of the proof we make usage of an intermediate space, denoted by $X_{a, b}$, to prove that  the minimizers  \eqref{min1} and \eqref{min2} belong to  $\mathcal{H}_{a, b}^1$. 

\noindent{\bf Definition of $X_{a, b}$.}\\
$\bullet$ If $(a, b)\in \mathcal{A}$ then
\begin{multline}
  X_{a, b}:=\Big\{u \in C^\infty (\rr^N\setminus\{0\}) \textrm{ with } \|u\|_{\mathcal{H}_{a, b}^1}<\infty \ |  \exists\  C>0 \textrm{ and } \exists \ D: S^{N-1}\rightarrow \rr \textrm{ s.t. }  \\
  x\mapsto D\left(\frac{x}{|x|}\right)\in C^1(\rr^N\setminus\{0\}) \textrm{ and } |u(x)|\leq \left|D\left(\frac{x}{|x|}\right)\right|e^{-C|x|^{b+1-a}}, \forall x\in \rr^N\setminus\{0\}   \Big\}.
\end{multline}
$\bullet$ If $(a, b)\in \mathcal{B}$ then
\begin{multline}
  X_{a, b}:=\Big\{u \in C^\infty (\rr^N\setminus\{0\}) \textrm{ with } \|u\|_{\mathcal{H}_{a, b}^1}<\infty \ |  \exists\  C>0 \textrm{ and } \exists \ D: S^{N-1}\rightarrow \rr \textrm{ s.t. }  \\
  x\mapsto D\left(\frac{x}{|x|}\right)\in C^1(\rr^N\setminus\{0\}) \textrm{ and } 
  |u(x)|\leq \left|D\left(\frac{x}{|x|}\right)\right||x|^{2(b+1)-N}e^{-C|x|^{b+1-a}}, \forall x\in \rr^N\setminus\{0\}   \Big\}.
\end{multline}
Notice that $(X_{a, b}, \|\cdot \|_{\mathcal{H}_{a, b}^1})$ is a normed space. The main result here is
\begin{proposition}\label{prop1}
  The space $C_0^\infty(\rr^N\setminus\{0\})$ is densely embedded in $(X_{a, b}, \|\cdot \|_{\mathcal{H}_{a, b}^1})$ provided $(a, b)\in \mathcal{A}\cup \mathcal{B}$. 
\end{proposition}

\begin{remark} \em
  As a consequence of Proposition \ref{prop1} we have that $X_{a, b}\subset \mathcal{H}_{a,b}^1$.	  This implies
  \begin{equation}
    \tilde{C}(N, a, b)\geq \inf_{u\in X_{a, b}}\tilde{E}(u)\geq \inf_{u\in \mathcal{H}_{a, b}^1} \tilde{E}(u)=\tilde{C}(N, a, b). 
  \end{equation}
  Therefore $\tilde{C}(N, a, b)=\inf_{u \in X_{a, b}}\tilde{E}(u)$ and the minimizers are actually achieved in $X_{a, b}$. 
\end{remark}

\begin{proof}[Proof of Proposition \ref{prop1}]

  In what follows,  to simplify the computations, we will write $\lesssim$ instead of $\leq C$ when we refer to universal constants C.

  For $\e>0$ small enough we consider the sequence $\{\eta_\e\}_\e$ defined on $\rr^N$ by
  \begin{equation}\label{apr_seq}
    \eta_\e(x)=\left\{\begin{array}{cc}
      1, & \e\leq |x|\leq 1/\e \\ [5pt]
      \frac{\log (|x|/\e^2)}{\log 1/\e}, & \e^2\leq |x|\leq \e\\ [5pt]
      \frac{\log \e |x|}{\log 1/\e},  & 1/\e\leq |x|\leq 1/\e^2\\ [5pt]
      0, & \textrm{ otherwise}.
    \end{array}\right. 
  \end{equation}
  We get that $\eta_\e\in C_c(\rr^N)\cap H^1(\rr^N)$ with $\|\eta_{\e}\|_{\mathcal{H}_{a, b}^1}< \infty$. The gradient satisfies the simplified form 
  \begin{equation}\label{grad}
    \nabla \eta_\e(x)=\left\{\begin{array}{cc}
      (\log \frac{1}{\e})^{-1}\frac{x}{|x|^2}, & \e^2< |x|< \e \textrm{ or } 1/\e< |x|< 1/\e^2  \\ [5pt]
      0, & \textrm{ otherwise}.
    \end{array}\right.
  \end{equation}
  Also, consider $$\rho_\e(x):=\left(\frac{\e^2}{2}\right)^{-N}\rho\left(\frac{2x}{\e^2}\right), \quad \textrm{ with } \int_{\rr^N} \rho_\e\dx =1, \quad \forall \e>0,$$
  where $\rho$ is the standard convolution kernel (mollifier) with $\int_{\rr^N}\rho dx=1$. 
  Notice that $\rho_\e\in C_0^\infty(\rr^N)$ with $\textrm{Supp} \rho_\e=B_{\e^2/2}(0)$ (i.e. the ball of radius $\e^2/2$ centered at $0$).  

  Let $u\in X_{a, b}$ and consider the approximation 
  \begin{equation}\label{reg_apr_seq}
    u_\e(x):=\phi_\e(x)u(x), \textrm{ where } \phi_\e(x):=(\rho_\e\ast \eta_\e)(x).
  \end{equation}
  Taking into account the convolution properties it holds that 
  $\phi_\e\in C_0^\infty(\rr^N\setminus\{0\})$ with $\textrm{Supp}\phi_\e\subset \{x\in \rr^N \ | \ \e^2/2\leq |x|\leq 1/\e^2+\e^2/2 \}$.

  We show that $\{u_\epsilon\}_{\epsilon>0} \subset C_0^\infty(\rr^N\setminus \{0\})$ approximates $u$ in the $\|\cdot\|_{\mathcal{H}_{a, b}^1}$-norm. 

  First, we have 
  $$D_j \phi_\e(x)=\rho_\e\ast D_j \eta_\e(x), \quad D_j(u-u_\e)=(1-\phi_\e) D_j u-uD_j \phi_\e,$$ 
  where $D_{j}=\p/\p x_{j}$.

  \noindent{\it Step 1:} $u_\e|x|^{-a}\to u|x|^{-a}$ in $L^2(\rr^N)$. 
  First we notice that 
  \begin{itemize}
    \item $\phi_\e\to 1$ a.e. as $\e \to 0$;
    \item $0\leq \phi_\epsilon(x)\leq 1$ for all $x\in \rr^N$.
  \end{itemize}
  Indeed, for the second item we have 
  \begin{align*}
    \phi_\e(x)= \int_{\R^{N}} \rho_\e(x-y)\eta_\e(y) dy \leq \int_{\R^{N}} \rho_\e(x-y) =\int_{\rr^N}\rho_\e(z)dz =1, \quad \forall x\in \rr^N.
  \end{align*}
  For the first item we observe that if $x\neq 0$ then for $\e$ small enough we have $\eta_\e=1$ in the support of $\rho_{\e}(x-\cdot)$ and so $\phi_\e(x)=1.$ 
  For instance, take $\e < \min \left\{\frac{2|x|}{3},\frac{1}{1+|x|} \right\}$.

  Therefore, we have
  \begin{align*}
    \|u_\e |x|^{-a} - u |x|^{-a}\|_{L^2(\rr^N)} =\|u|x|^{-a}(1-\phi_\e)\|_{L^2(\rr^N)} \to 0, \textrm{ as } \e \to 0,
  \end{align*}
  by the dominated convergence theorem. 

  \noindent{\it Step 2: } $|x|^{-b}D_j u_{\e}\to |x|^{-b}D_j u$ in $L^2(\rr^N)$. 

  We have $D_j(u_\e-u)|x|^{-b}=(1-\phi_\e) |x|^{-b} D_j u-D_j\phi_\e |x|^{-b} u $. Clearly, $(1-\phi_\e) |x|^{-b} D_j u\to 0$ in $L^2(\rr^N)$ again by dominated convergence theorem. 

  It remains to show that $D_j\phi_\e |x|^{-b} u\to 0$ in $L^2(\rr^N)$. 

  Let us denote $\mathbb{A}_\e:=\{x\in \rr^N \ |\ \e^2< |x|< \e\}$ and $\mathbb{B}_\e:=\{x\in \rr^N \ |\ 1/\e < |x|< 1/\e^2\}$. 

  By \eqref{grad} and Cauchy-Schwarz inequality we first have  
  \begin{align*}
    |\rho_e\ast D_j \eta_\e(x)|&=\left|\left(\log \frac{1}{\e}\right)^{-1} \int_{\mathbb{A}_\e\cup \mathbb{B}_\e} \rho_\e(x-y)\frac{y_j}{|y|^2} dy  \right|\\
    &\leq \left(\log \frac{1}{\e}\right)^{-1} \left(\int_{\mathbb{A}_\e\cup \mathbb{B}_\e} \rho_\e(x-y) dy\right)^{1/2}\left(\int_{\mathbb{A}_\e\cup \mathbb{B}_\e} \rho_\e(x-y)\frac{1}{|y|^2} dy\right)^{1/2}\\
    & \leq \left(\log \frac{1}{\e}\right)^{-1} \left(\int_{\mathbb{A}_\e\cup \mathbb{B}_\e} \rho_\e(x-y)\frac{1}{|y|^2} dy\right)^{1/2}   
  \end{align*} 
  Then, by Fubini we obtain  
  \begin{align*}
    \int_{\R^{N}}|x|^{-2b}|u(x)|^2 |D_j \phi_\e|^2 dx &=\int_{\rr^N} |x|^{-2b}|u(x)|^2|\rho_e\ast D_j \eta_\e(x)|^2 dx \\
    &\leq  \int_{\rr^N} |x|^{-2b} |u(x)|^2 \left(\log \frac{1}{\e}\right)^{-2}  \int_{\mathbb{A}_\e\cup \mathbb{B}_\e} \rho_\e(x-y)\frac{1}{|y|^2} dy dx\\
    &=\left(\log \frac{1}{\e}\right)^{-2} \int_{\mathbb{A}_\e} \frac{1}{|y|^2}\int_{\rr^N} |x|^{-2b} |u(x)|^2 \rho_e(x-y) dx dy\\
    & + \left(\log \frac{1}{\e}\right)^{-2} \int_{\mathbb{B}_\e} \frac{1}{|y|^2}\int_{\rr^N} |x|^{-2b} |u(x)|^2 \rho_e(x-y) dx dy\\
    &:=I_{1, \e} + I_{2, \e}.   
  \end{align*}
  Since for any $y\in \mathbb{A}_\epsilon$ and $x\in B_{\epsilon^2/2}(y)$ we have $\frac 3  2 |y|\geq |x|\geq \frac 1 2 |y|$ and we obtain 
  \begin{align}\label{I1-est}
    I_{1, \epsilon} & \lesssim \left(\log \frac{1}{\e}\right)^{-2} \int_{\mathbb{A}_\e} |y|^{-2-2b} \int_{B_{\epsilon^2/2}(y)} |u(x)|^2 \rho_\e(x-y) dx dy 
  \end{align}
  Also,	since for any $y\in \mathbb{B}_\e$ and $x\in B_{\epsilon^2/2}(y)$ we have $\frac 3  2 |y|\geq |x|\geq \frac 1 2 |y|$ for $\epsilon$ small enough, we obtain
  \begin{align}\label{I2-est}
    I_{2, \epsilon} & \lesssim \left(\log \frac{1}{\e}\right)^{-2} \int_{\mathbb{B}_\e} |y|^{-2-2b} \int_{B_{\epsilon^2/2}(y)} |u(x)|^2 \rho_\e(x-y) dx dy 
  \end{align}
  Next we distinguish four cases for $(a, b)$: 
  \begin{enumerate}
    \item The case $(a, b)\in \mathcal{A}_1$. 
      Then, there holds
      \begin{align*}
	I_{1, \epsilon} &\lesssim\|D\|_{L^\infty(S^{N-1})}^2 \left(\log \frac{1}{\e}\right)^{-2}  \int_{\mathbb{A}_\e} |y|^{-2-2b} \int_{B_{\epsilon^2/2}(y)} \rho_\e(x-y) dx dy\\
	& \lesssim\left(\log \frac{1}{\e}\right)^{-2}  \int_{\mathbb{A}_\e} |y|^{-2-2b} dy\\
	& \lesssim\left(\log \frac{1}{\e}\right)^{-2} \int_{\e^2}^{\e} r^{-2b+N-3}dr\\
	& \lesssim\left(\log \frac{1}{\e}\right)^{-2} \int_{\e^2}^{\e} \frac 1 r dr\\ 
	&= \left(\log \frac 1 \epsilon\right)^{-1}\rightarrow 0, \textrm{ as } \e \to 0.  \\
	I_{2, \epsilon} &\lesssim\|D\|_{L^\infty(S^{N-1})}^2 \left(\log \frac{1}{\e}\right)^{-2}  \int_{\mathbb{B}_\e} |y|^{-2-2b} e^{-\tilde{C}|y|^{b+1-a}} \int_{B_{\epsilon^2/2}(y)} \rho_\e(x-y) dx dy\\
	& \lesssim\left(\log \frac{1}{\e}\right)^{-2}  \int_{\mathbb{B}_\e} |y|^{-2-2b} e^{-\tilde{C}|y|^{b+1-a}} dy\\
	& \lesssim\left(\log \frac{1}{\e}\right)^{-2} \int_{1/\e}^{1/\e^2} r^{-2b+N-3} e^{-\tilde{C}r^{b+1-a}} dr\\
	& \lesssim\left(\log \frac{1}{\e}\right)^{-2} e^{-\tilde{C}\left(\frac 1 \epsilon\right)^{b+1-a}} \int_{1/\e}^{1/\e^2} r^{-2b+N-3}  dr \to 0, \textrm{ as } \epsilon \to 0. 
      \end{align*}
    \item The case $(a, b)\in \mathcal{A}_2$. 
      Then, there holds
      \begin{align*}
	I_{1, \epsilon} &\lesssim\|D\|_{L^\infty(S^{N-1})}^2 \left(\log \frac{1}{\e}\right)^{-2}  \int_{\mathbb{A}_\e} |y|^{-2-2b} e^{-\tilde{C}|y|^{b+1-a}} \int_{B_{\epsilon^2/2}(y)} \rho_\e(x-y) dx dy\\
	& \lesssim\left(\log \frac{1}{\e}\right)^{-2}  \int_{\mathbb{A}_\e} |y|^{-2-2b}  e^{-\tilde{C}|y|^{b+1-a}} dy\\
	& \lesssim\left(\log \frac{1}{\e}\right)^{-2} e^{-\tilde{C}\epsilon^{b+1-a}} \int_{\e^2}^{\e} r^{-2b+N-3}dr\rightarrow 0, \textrm{ as } \e \to 0.  \\
	I_{2, \epsilon} &\lesssim\|D\|_{L^\infty(S^{N-1})}^2 \left(\log \frac{1}{\e}\right)^{-2}  \int_{\mathbb{B}_\e} |y|^{-2-2b}  \int_{B_{\epsilon^2/2}(y)} \rho_\e(x-y) dx dy\\
	& \lesssim\left(\log \frac{1}{\e}\right)^{-2}  \int_{\mathbb{B}_\e} |y|^{-2-2b} dy\\
	& \lesssim\left(\log \frac{1}{\e}\right)^{-2} \int_{1/\e}^{1/\e^2} r^{-2b+N-3} dr\\
	& \lesssim\left(\log \frac{1}{\e}\right)^{-2}  \int_{1/\e}^{1/\e^2} \frac{1}{r}  dr \to 0,\\
	&= \left(\log \frac 1 \epsilon\right)^{-1}\rightarrow 0, \textrm{ as } \e \to 0. 
      \end{align*}
    \item The case $(a, b)\in \mathcal{B}_1$.
      Then, there holds
      \begin{align*}
	I_{1, \epsilon} &\lesssim\|D\|_{L^\infty(S^{N-1})}^2 \left(\log \frac{1}{\e}\right)^{-2}  \int_{\mathbb{A}_\e} |y|^{-N} e^{-\tilde{C}|y|^{b+1-a}} \int_{B_{\epsilon^2/2}(y)} \rho_\e(x-y) dx dy\\
	& \lesssim\left(\log \frac{1}{\e}\right)^{-2}  \int_{\mathbb{A}_\e} |y|^{-N}  e^{-\tilde{C}|y|^{b+1-a}} dy\\
	& \lesssim\left(\log \frac{1}{\e}\right)^{-2} e^{-\tilde{C}\epsilon^{b+1-a}} \int_{\e^2}^{\e} \frac 1 r dr\\
	& \lesssim\left(\log \frac{1}{\e}\right)^{-1} e^{-\tilde{C}\epsilon^{b+1-a}} \rightarrow 0, \textrm{ as } \e \to 0.  \\
	I_{2, \epsilon} &\lesssim\|D\|_{L^\infty(S^{N-1})}^2 \left(\log \frac{1}{\e}\right)^{-2}  \int_{\mathbb{B}_\e} |y|^{-N}  \int_{B_{\epsilon^2/2}(y)} \rho_\e(x-y) dx dy\\
	& \lesssim\left(\log \frac{1}{\e}\right)^{-2}  \int_{\mathbb{B}_\e} |y|^{-N}   dy\\
	& \lesssim\left(\log \frac{1}{\e}\right)^{-2} \int_{1/\e}^{1/\e^2} \frac 1 r dr\\
	& \lesssim\left(\log \frac{1}{\e}\right)^{-1} \rightarrow 0, \textrm{ as } \e \to 0.  
      \end{align*}
    \item The case $(a, b)\in \mathcal{B}_2$.
      Then, there holds
      \begin{align*}
	I_{1, \epsilon} &\lesssim\|D\|_{L^\infty(S^{N-1})}^2 \left(\log \frac{1}{\e}\right)^{-2}  \int_{\mathbb{A}_\e} |y|^{-N}  \int_{B_{\epsilon^2/2}(y)} \rho_\e(x-y) dx dy\\
	& \lesssim\left(\log \frac{1}{\e}\right)^{-2}  \int_{\mathbb{A}_\e} |y|^{-N}   dy\\
	& \lesssim\left(\log \frac{1}{\e}\right)^{-2} \int_{\e^2}^{\e} \frac 1 r dr\\
	& \lesssim\left(\log \frac{1}{\e}\right)^{-1} \rightarrow 0, \textrm{ as } \e \to 0.  \\
	I_{2, \epsilon} &\lesssim\|D\|_{L^\infty(S^{N-1})}^2 \left(\log \frac{1}{\e}\right)^{-2}  \int_{\mathbb{B}_\e} |y|^{-N} e^{-\tilde{C}|y|^{b+1-a}} \int_{B_{\epsilon^2/2}(y)} \rho_\e(x-y) dx dy\\
	& \lesssim\left(\log \frac{1}{\e}\right)^{-2}  \int_{\mathbb{B}_\e} |y|^{-N}  e^{-\tilde{C}|y|^{b+1-a}} dy\\
	& \lesssim\left(\log \frac{1}{\e}\right)^{-2} e^{-\tilde{C}\left(\frac{1}{\e}\right)^{b+1-a}} \int_{1/\e}^{1/\e^2} \frac 1 r dr\\
	& \lesssim\left(\log \frac{1}{\e}\right)^{-1} e^{-\tilde{C}\left(\frac{1}{\e}\right)^{b+1-a}} \rightarrow 0, \textrm{ as } \e \to 0.  \\
      \end{align*}
  \end{enumerate}
\end{proof}

  \end{document}